\documentclass{amsart}  %[Font size]

\title{An alternate description of equivariant connections}
\author{Corbett Redden}
\address{Department of Mathematics, LIU Post, Long Island University, 720 Northern Blvd, Brookville, NY 11548, USA}
\email{corbett.redden@liu.edu}
\date{July 17, 2017}
\usepackage{hyperref} %Hyper-links between references in paper
\hypersetup{ 
colorlinks,% 
citecolor=black,% 
filecolor=black,% 
linkcolor=black,% 
urlcolor=black } %Sets all hyperlinks to appear as normal black text.
 
\usepackage{stmaryrd}  %\longmapsto, \llbracket
\usepackage{amssymb}
\usepackage{mathtools}   % \xhookrightarrow
\usepackage[cal=boondox, calscaled=1.03]{mathalfa}  %used for the boondox script Q
\usepackage{tikz-cd} %TikZ commutative diagram package
\numberwithin{equation}{section}

\theoremstyle{plain}
\newtheorem{thm}[equation]{Theorem}
\newtheorem{prop}[equation]{Proposition}
\newtheorem{lemma}[equation]{Lemma}

\theoremstyle{definition}

\newtheorem{rem}[equation]{Remark}

\DeclareFontFamily{OT1}{pzc}{}
\DeclareFontShape{OT1}{pzc}{m}{it}%
              {<-> s * [1.15] pzcmi7t}{}
\DeclareMathAlphabet{\mathpzc}{OT1}{pzc}{m}{it}

\newcommand{\R}{\mathbb{R}}

\newcommand{\fg}{\mathfrak{g}}
\newcommand{\iso}{\cong}
\renewcommand{\=}{:=}

\newcommand{\wt}{\widetilde}
\newcommand{\fk}{\mathfrak{k}}
\newcommand{\Id}{\operatorname{id}}

\newcommand{\cC}{\mathcal{C}}

\newcommand{\bas}{\operatorname{bas}}
\newcommand\Ad{\operatorname{Ad}}
\newcommand\op{\mathrm{op}}

\renewcommand{\o}{\circ}
\renewcommand{\d}{\partial}
\renewcommand\O{\Omega}

\newcommand{\cat[1]}{\mathsf{#1}}
\newcommand\Man{\cat[Man]}
\newcommand\Set{\cat[Set]}
\newcommand\Pre{\cat[Pre]}
\newcommand\Gpd{{\cat[Gpd]}}

\newcommand{\B}{\mathpzc{B}}
\newcommand{\E}{\mathpzc{E}}
\newcommand{\Bn}{\mathpzc{B_{\scriptscriptstyle{\nabla}}\mkern-2mu}}
\newcommand{\En}{\mathpzc{E_{\scriptscriptstyle{\nabla}}\mkern-2mu}}

\newcommand\GMan{\text{-}\Man}
\newcommand{\cQ}{\mathcal{Q}}
\newcommand{\cQn}{\mathcal{Q}_{\mkern-1mu\scriptscriptstyle{\nabla}}\mkern-1mu}
\newcommand{\Tn}{T^{\scriptscriptstyle{\nabla}}}
\newcommand\cM{\mathpzc{M}}

\newcommand\Bun[1][\T]{\cat[Bun]_{#1}}
\newcommand\Bunc[1][\T]{\cat[Bun]_{#1,\nabla}}
\newcommand\GBun[1][\T]{\text{-}\cat[Bun]_{#1}}
\newcommand\GBunc[1][\T]{\text{-}\cat[Bun]_{#1,\nabla}}

\renewcommand{\sslash}{/\!\!/}
\newcommand{\nslash}{ /\!\!/\mkern-4mu_{\scriptscriptstyle{\nabla}}}

\newcommand{\mycdot}{\mkern-2mu \cdot \mkern-2mu}

\begin{document}

%%%%%%%%%%%%%%%%%%%%%%%%%%%%%%%%%%%%%%%%%%%%%
%%%%%%%%%%%%%%%%%%%%%%%%%%%%%%%%%%%%%%%%%%%%%
% Top data

\begin{abstract}In this note, we consider a Lie group $G$ acting on a manifold $M$.  We prove that the category of principal bundles with connection on the differential quotient stack is equivalent to the category of $G$-equivariant principal bundles on $M$ with $G$-invariant connection.
\end{abstract}

\maketitle
\thispagestyle{empty} %No page number on first page

%%%%%%%%%%%%%%%%%%%%%%%%%%%%%%%%%%%%%%%%%%%%%
%%%%%%%%%%%%%%%%%%%%%%%%%%%%%%%%%%%%%%%%%%%%%
\section{Introduction}\label{Sec:Introduction}

Let $G, K$ be arbitrary finite-dimensional Lie groups, and let $M$ be a smooth finite-dimensional manifold with a $G$-action.  Let $\Bunc[G](X)$ denote the groupoid of principal $G$-bundles with connection on a manifold $X$.  This paper's result can be phrased as follows.  Consider all possible constructions $\cQn$ of the following form: to any $(P,\Theta) \in \Bunc[G](X)$, together with a $G$-equivariant map $P \xrightarrow{F} M$ is associated a $K$-bundle with connection $\cQn(P,\Theta, F) \in \Bunc[K](X)$; and, as indicated in the following diagram, to any map $P'\xrightarrow{\varphi}P$ of principal $G$-bundles is associated a morphism $\cQn(\varphi)$ of $K$-bundles with connection.
\[ \begin{tikzcd}[column sep=small]
(P', \varphi^*\Theta) \ar[r, "\varphi"] \ar[d] & (P,\Theta) \ar[r,"F"] \ar[d] & M \\
X' \ar[r, "\bar{\varphi}"] & X
\end{tikzcd} \quad \mapsto \quad
\begin{tikzcd}[column sep=tiny]\cQn(P', \varphi^*\Theta, F\o \varphi) \ar[rr, "\cQn(\varphi)"] \ar[dr]&&
 \bar{\varphi}^* \cQn(P,\Theta,F) \ar[dl] \\ & X'
\end{tikzcd}
 \]
% \[ \includegraphics{a1.pdf} \]
The collection of such constructions $\cQn$ naturally forms a category that we denote ${\Bunc[K](\En G \times_G M)}$.  It is the groupoid of $K$-bundles with connection on the differential Borel quotient stack $\En G \times_G M$.  Theorem \ref{Thm:EquivBun} states that there is a natural equivalence of categories
\[ \Bunc[K](\En G \times_G M) \iso G\GBunc[K](M),\]
where $G\GBunc[K](M)$ is the category of $G$-equivariant principal $K$-bundles on $M$ with $G$-invariant connection.  

A functor $G\GBunc[K](M) \to \Bunc[K](\En G \times_G M)$ was previously given in both \cite{Redden16a} and \cite{1606.01129}.  In this paper, we show that the functor is an equivalence.  To do so, we construct an explicit inverse functor in the following way.  Evaluate a given $\cQn$ on the trivial $G$-bundle $P=M\times G$, with trivial connection and map $F(m,g) = {g^{-1}  m}$.  The resulting $K$-bundle with connection naturally has the structure of a $G$-equivariant bundle; the $G$-action is induced by left multiplication on $G$, viewed as a gauge transformation of $M\times G \to M$, together with the functoriality of $\cQn$.

Here is a brief explanation of the symbols we use.  We consider the stack $\En G \times_G M \in \Pre_{\Gpd}$ as a presheaf of groupoids on the site of manifolds.  Likewise, $\Bn K \in \Pre_{\Gpd}$ is the stack of principal $K$-bundles with connection, and we define $\Bunc[K](\En G\times_G M) \= \Pre_{\Gpd}(\En G \times_G M, \Bn K)$ as morphisms in the category of prestacks.  The analogous structures defined without connections are denoted by omitting the symbol $\nabla$.  Just as the Lie groupoid $M \sslash G = (G\times M \rightrightarrows M)$ stackifies to $\E G \times_G M$, there is an analogous $M \nslash G \in \Pre_{\Gpd}$ that stackifies to $\En G \times_G M$ and induces a natural equivalence of groupoids $\Pre_{\Gpd}(\En G\times_G M, \Bn K) \iso \Pre_{\Gpd}(M \nslash G, \Bn K)$.  In the remainder of this introduction, we will favor writing $M\nslash G$ over $\En G \times_G M$ simply because it uses less space.

The notion that $G$-equivariant $K$-bundles on $M$ correspond to $K$-bundles on $M\sslash G$ is well-established.  Previous related work includes but is certainly not limited to \cite{MR2977576, MR2817778, MR2147734, MR2270285, MR2506770, 1510.06392, 1310.7930}.  However, one must be careful when adding connections to the picture.  We want to emphasize the point that, when considering $K$-connections, one should also include $G$-connections in the quotient stack.  To see this explicitly, we express Theorem \ref{Thm:EquivBun} as the following diagram.  
\begin{equation}\label{Eq:ConnectionDiag} \begin{tikzcd}
M \arrow[rr, "{(Q,\Theta)}"] \arrow[d]&&[50]\Bn K \arrow[d] \\
M \nslash G \arrow[r] \arrow[urr, dotted, "\Theta \ G\text{-invt}"]& M \sslash G \arrow[ur,dotted, "\Theta \ G\text{-bas}" '] \arrow[r,dotted, "G\text{-struct on }Q" '] & \B K 
\end{tikzcd} \end{equation}
%\begin{equation}\label{Eq:ConnectionDiag} \raisebox{-0.42\height}{\includegraphics{a2.pdf}} \end{equation}
Given $(Q,\Theta) \in \Bunc[K](M)$, a $G$-structure on $Q$ is equivalent to the map $M\sslash G \to \B K$.  The map $M\to \Bn K$ factors through $M\nslash G \to \Bn K$ if and only if $\Theta$ is $G$-invariant, and it factors through $M\sslash G \to \Bn K$ if and only if $\Theta$ is $G$-basic.

The importance of using $M\nslash G$, as opposed to $M\sslash G$, becomes clear when considering geometric objects such as characteristic forms.  A bundle $(Q,\Theta) \in G\GBunc[K](M)$ induces a map $M\nslash G \to \Bn K$.  It was shown in \cite{Redden16a, 1606.01129} that, under the natural isomorphism $\O(M\nslash G) \iso \O_G(M)$ from \cite{MR3049871}, the induced map
\[ (S^n\fk^*)^K \iso \O^{2n}(\Bn K) \xrightarrow{(Q,\Theta)_G} \O^{2n}(M\nslash G) \iso \O^{2n}_G(M) \]
recovers the classical equivariant Chern--Weil forms of Berline--Vergne \cite{MR705039}.  Specifically, one recovers the actual form in $\left(S\fg^* \otimes \O(M)\right)^G$, as opposed to merely the cohomology class in $H_G^*(M;\R)$.

The paper is structured in the following way.  Section \ref{Sec:Review} explains the various notations and conventions used in the paper.  Section \ref{Sec:EquivConns} gives a characterization of $G$-equivariant connections on $M$ in terms of bundle isomorphisms over $G\times M$.  While much of this is known in various forms, we could not find any sources dealing with the connections in the form we needed.  Section \ref{Sec:Equivalences} is the heart of the paper.  Through a series of lemmas, we construct the inverse functor $\Bunc[K](M\nslash G) \to G\GBunc[K](M)$ and prove it is an equivalence of categories.  We should note that Section \ref{Sec:Equivalences} takes place in the world of prestacks, or functors $\Man^\op \to \Gpd$.  Because all constructions are essentially geometric, we do not need to invert weak equivalences in any way, and we do not require any special techniques from stacks or simplicial sheaves.  

The author would like to thank Byungdo Park and an anonymous referee for many helpful comments and suggestions to improve the paper.
%%%%%%%%%%%%%%%%%%%%%%%%%%%%%%%%%%%%%%%%%%%%%
%%%%%%%%%%%%%%%%%%%%%%%%%%%%%%%%%%%%%%%%%%%%%
\section{Preliminaries}\label{Sec:Review}

We begin by explaining our various notations and conventions.  See \cite{Redden16a} for further details.  Let $\Man$ denote the category of finite-dimensional smooth manifolds and smooth maps.  For $G$ a Lie group, $G\GMan$ is the category of $G$-manifolds (manifolds equipped with a smooth $G$-action) and equivariant maps.

For $M\in \Man$, we use $\Bunc[K](M)$ to denote the groupoid (category in which every morphism is invertible) of principal $K$-bundles with connection on $M$; morphisms are bundle isomorphisms that preserve the connection.  For $M \in G\GMan$, let $G\GBunc[K](M)$ denote the groupoid of $G$-equivariant principal $K$-bundles on $M$ with $G${\it -invariant connection}.  Explicitly, an object of $G\GBunc[K](M)$ is some $(Q,\Theta) \in \Bunc[K](M)$, plus the structure of $Q \in G\GMan$ such that: the actions of $G$ and $K$ on $Q$ commute, $Q \overset{\pi}\to M$ is $G$-equivariant, and $g^*\Theta =\Theta \in \O^1(Q;\fk)$ for all $g\in G$.  For $(Q_i, \Theta_i) \in G\GBunc[K](M)$, a morphism $(Q_1,\Theta_1) \overset{\phi}\to (Q_2,\Theta_2)$ is a $G$-equivariant map $Q_1 \overset{\phi}\to Q_2$, covering the identity map $\Id_M$,  that is an isomorphism of $K$-bundles with connection.  

A $G$-action on a manifold $Q$ induces a Lie algebra homomorphism $\fg \xrightarrow{\rho} \Gamma(TQ)$, defined for $q \in Q$ and $A \in TG_e \iso \fg$ by
\begin{equation}\label{Eq:VF}  \rho_{q}(A) = \begin{cases} \frac{d}{dt}\big|_{_{t=0}} \, e^{-tA}  \,q & \text{(left action)} \\
 \frac{d}{dt}\big|_{_{t=0}} \, q \, e^{tA} & \text{(right action)}  \\
\end{cases} \, \in TQ_q. \end{equation}
For $(Q,\Theta) \in G\GBunc[K](M)$ we say the connection is $G${\it -basic} if it is $G$-invariant and satisfies the additional property $\iota_{\rho(A)} \Theta = 0$ for all $A \in \fg$.  These bundles with connection are the objects of $G_{\bas} \GBunc[K](M)$, which is a full subgroupoid of $G\GBunc[K](M)$.  

The groupoids $\Bun[K](M)$ and $G\GBun[K](M)$ are defined analogously, using principal $K$-bundles with no specified connection.

\begin{rem}Our convention is to have the structure group of a principal bundle (usually $K$) act by right multiplication, while the additional symmetry group (usually $G$) acts by left multiplication.  When necessary, we use inverses to switch between left and right actions, which does not change the induced map $\rho$. \end{rem}

A presheaf of groupoids on the site of manifolds $\cM \in \Pre_\Gpd$, commonly known as a prestack, is a (pseudo-)functor $\cM\colon \Man^\op \to \Gpd$.  Associated to any test manifold $X$ is a groupoid $\cM(X)$, and a smooth map $X_1 \xrightarrow{f} X_2$ induces a functor $\cM(X_2) \xrightarrow{f^*} \cM(X_1)$.  There are associative natural transformations $f^* g^* \iso (g\o f)^*$ to deal with composition of functions; we show them when necessary but do not give them special names.  There is an important full subcategory of stacks within prestacks, but we do not require any of the tools associated to stacks.  See Remark \ref{Rem:Stacks} for a further comment on this.

The collection of morphisms between prestacks naturally forms a groupoid, making $\Pre_\Gpd$ into a 2-category.  An object $\cQ$ in the groupoid $\Pre_\Gpd(\cM_1, \cM_2)$ is a 1-morphism $\cM_1\xrightarrow{\cQ} \cM_2$, which is a collection of functors $\cM_1(X) \xrightarrow{\cQ_X} \cM_2(X)$ for all $X$, together with natural transformations $\cQ_f \colon \cQ_X \o f^* \Rightarrow f^* \o \cQ_Y$ for all smooth $f \colon X \to Y$.   A morphism $\cQ \xrightarrow{\varphi} \cQ'$ in $\Pre_\Gpd(\cM_1, \cM_2)$, which is a 2-morphism 
$\begin{tikzcd}[column sep=huge]
{\cM_1}  & {\cM_2}
 \ar[from=1-1, to=1-2, bend left=20, "" '{name=U}, "\cQ"]   \ar[from=1-1, to=1-2, bend right=20, "" {name=D}, "\cQ' " '] 
\ar[Rightarrow, from=U, to=D, "\varphi" '] 
\end{tikzcd}$  
in $\Pre_\Gpd$, is a collection of natural transformations $\varphi_X \colon \cQ_X \Rightarrow \cQ'_X$ compatible with the pullback structures.  Any $\varphi$ must be invertible since $\cM_2(X)$ is always a groupoid.  We later omit the subscripts on $\cQ$ and $\varphi$, as they are easily determined from context.

We briefly define the examples relevant to this paper.  We refer the reader to \cite{MR3049871, Redden16a} for further details and explanations of these particular objects.  Any manifold $M \in \Man$  defines a stack by $M(X) \= C^\infty(X,M) \in \Set \subset \Gpd$, regarded as a category with no nontrivial morphisms.  The Yoneda Lemma states that this is a faithful embedding, and there is a natural equivalence of categories $\cM(X) \iso \Pre_\Gpd(X, \cM)$ for any $\cM \in \Pre_\Gpd$.  
Define $\Bn K$ as the stack that assigns to every test manifold $X$ the groupoid of principal $K$-bundles with connection, $\Bn K (X) \= \Bunc[K](X)$.  We use the natural Yoneda equivalence $\Bunc[K](X) \iso \Pre_\Gpd(X, \Bn K)$ to make the following definition/abbreviation for $\cM \in \Pre_\Gpd$,
\begin{equation} \Bunc[K](\cM) \= \Pre_\Gpd(\cM, \Bn K). \end{equation}

For $M \in G\GMan$, we have the {\it differential Borel quotient stack} $\En G \times_G M$.  An object $(P,\Theta,F) \in (\En G \times_G M)(X)$ is a principal $G$-bundle with connection $(P,\Theta)\to X$, together with a $G$-equivariant map $P\xrightarrow{F} M$.  Morphisms are given by connection-preserving bundle morphisms that are compatible with the maps to $M$.  We draw objects and morphisms in the following two equivalent ways, where $\varphi^*\Theta_2=\Theta_1$.
\[ \begin{tikzcd}[column sep=huge]
{X}  & {\En G \times_G M}
 \ar[from=1-1, to=1-2, bend left=40, "" '{name=U}, "{(P_1, \Theta_1, F_1)}"]   \ar[from=1-1, to=1-2, bend right=40, "" {name=D}, "{(P_2, \Theta_2, F_2)}" '] 
\ar[Rightarrow, from=U, to=D, "\varphi" '] 
\end{tikzcd}
 \quad \leftrightsquigarrow \quad 
\begin{tikzcd}[row sep=small] 
 & (P_1, \Theta_1) \ar[dl] \ar[dr, "F_1"] \ar[dd, "\varphi", "\iso" ', swap] \\
 X && M \\
 &(P_2, \Theta_2) \ar[ul] \ar[ur, "F_2" '] 
\end{tikzcd} \]
%\[ \includegraphics{a3.pdf} \]
When $P=X\times G$ is the trivial bundle, the connection $\Theta$ and equivariant map $F$ are uniquely determined by their restriction to $X$ under the canonical frame.   This gives a map $M\nslash G \to \En G \times_G M$, as explained in the following paragraph.

In general, given a (left) $G$-action on a set $S$, let $S \sslash G = (G \times S \rightrightarrows S)$ denote the associated action groupoid.  For $M\in G\GMan$, the smooth structures make $M \sslash G$ into a Lie groupoid, which in turn defines the prestack 
\[ M \sslash G \= \Big( G\times M \overset{d_0}{\underset{d_1}{\rightrightarrows}} M \Big)  \in \Pre_\Gpd\]
via the Yoneda embedding.  Explicitly, the groupoid $(M\sslash G)(X)$ has $C^\infty(X,M)$ as its set of objects and $C^\infty(X,G\times M)$ as its set of morphisms.  Let $\O^1(\fg) \in \Pre_\Gpd$ assign to any test manifold the set of $\fg$-valued 1-forms $\O^1(\fg)(X) \= \O^1(X;\fg)$.  The action of gauge transformations on connections defined in a local frame induces a natural (right) action of the sheaf $G$ on $\O^1(\fg)$. This is given by the familiar formula
\[ \alpha \mycdot g \= \Ad_{g^{-1}} \alpha + g^*\theta_\fg,\]
where $\alpha \in \O^1(X;\fg)$, $g \in C^\infty(X,G)$, and $\theta_\fg \in \O^1(G;\fg)$ is the Maurer--Cartan 1-form (usually written $g^{-1}dg$ for matrix groups).   Define the {\it differential action groupoid} by 
\begin{equation}\label{Eq:DiffActionGpd}
M \nslash G \= (\O^1(\fg) \times M) \sslash G =  \Big(G \times \O^1(\fg) \times M
 \overset{\d_0}{\underset{\d_1}{\rightrightarrows}} \O^1(\fg) \times M \Big) \in \Pre_\Gpd,
\end{equation}
with $\d_0(g, \alpha, f) = (\alpha, f)$ and $\d_1(g, \alpha, f) = (\alpha \mycdot g^{-1}, gf)$.  The trivial bundle  defines a natural map $M \nslash G \to \En G \times_G M$, which we conveniently write
\[ \begin{tikzcd}
{X}  & [20]M \nslash G   \ar[r] &[-10] \En G \times_G M 
 \ar[from=1-1, to=1-2, bend left=40,"" '{name=U}, "{(\alpha,\, f)}"]   
 \ar[from=1-1, to=1-2, bend right=40, "" {name=D}, "{(\alpha \cdot g^{-1}, \,g f)}" '] 
\ar[Rightarrow, from=U, to=D, "{(g, \alpha, f)}" description] 
\end{tikzcd}  \rightsquigarrow 
\begin{tikzcd}[row sep=small, column sep=small] 
 & (X\times G, \Theta_{\alpha}) \ar[dl] \ar[dr, "F_f"] \ar[dd, "\varphi_g", "\iso" ', swap] \\
 X && M, \\
 &(X \times G, \Theta_{\alpha\cdot g^{-1}}) \ar[ul] \ar[ur, "F_{gf}" '] 
\end{tikzcd} \]
%\[ \includegraphics{a4.pdf} \]
where $F_f(x, h) = h^{-1} f(x)$, $\varphi_g(x,h) = (x, g(x)h)$, and $(\Id_X \times \, e)^*\Theta_\alpha = \alpha$.

When $X$ is contractible, then $(M\nslash G) (X) \to (\En G \times_G M)(X)$ is an equivalence of categories, and the choice of a global section provides an explicit inverse.

\begin{rem}\label{Rem:Stacks}The map $M\nslash G \to \En G \times_G M$ is actually the stackification map, giving an equivalence of the induced stack $[M\nslash G] \iso \En G \times_G M$.  Since $\Bn K$ is a stack, this implies $\Tn_2$ from Section \ref{Sec:Equivalences} is an equivalence
\begin{equation}\label{Eq:StackificationBundles} \Pre_{\Gpd} \left( \En G \times_G M, \Bn K\right) \overset{\iso}\longrightarrow \Pre_\Gpd \left( M \nslash G, \Bn K \right),\end{equation}
making Lemma \ref{Lem:T2Faithful} unnecessary (see \cite{MR3019405} for many of these details).  However, we include the lemma and work only with prestacks for two reasons.  First, distinguishing between structures on $M\nslash G$ and $\En G \times_G M$ avoids ambiguities in notation.  Second, we want to emphasize that only functorial geometric constructions are needed, as opposed to more sophisticated stack machinery.
\end{rem}

%%%%%%%%%%%%%%%%%%%%%%%%%%%%%%%%%%%%%%%%%%%%%%%%%%%
%%%%%%%%%%%%%%%%%%%%%%%%%%%%%%%%%%%%%%%%%%%%%%%%%%%
\section{Equivariant Connections}\label{Sec:EquivConns}
We proceed to characterize equivariant bundles with connection in terms of bundles on the Lie groupoid $G\times M \rightrightarrows M$.  Proposition \ref{Prop:EquivStruct} is certainly well-known, and variations on   
Proposition \ref{Prop:EquivConns}, with $K$ abelian or involving pseudoconnections, appear in the literature.  Examples include \cite[Appendix]{Brylinski-Grb00}, \cite[\textsection 4]{MR2817778}, \cite[\textsection 5]{MR2147734}, \cite[p. 10]{MR2206877}, \cite{MR2270285}, \cite{MR2506770}.  However, we could not find a detailed statement of our Proposition \ref{Prop:EquivConns}, or a detailed proof Proposition \ref{Prop:EquivStruct} in the same form as what we use here.  Since the correspondence is essential to Theorem \ref{Thm:EquivBun}, we go ahead and include an explicit account. 

%As mentioned in Section \ref{Sec:Introduction}, there is a well-known correspondence between $G$-equivariant bundles on $M$ and bundles on $M \sslash G$; see for example  \cite{MR2270285}, \cite[Introduction]{MR2506770}, and \cite[p. 10]{MR2206877}.  While Proposition \ref{Prop:EquivStruct}
%
%\reminder{
%(Reference? This is noted in Introduction to \cite{MR2506770}, as well as p.10 of \cite{MR2206877}, but without many details or citation.  Look carefully at  Example 2.34 of \cite{MR2119241}) }

The Lie groupoid $G\times M \rightrightarrows M$ extends to a simplicial manifold $\{G^\bullet \times M\}$, and we describe the relevant face maps $d_i$ and identities through the following diagram.
\[ \begin{tikzcd} G\times G \times M \ar[r] \ar[r, shift left=2] \ar[r, shift right=2] & G \times M \ar[r, shift left] \ar[r, shift right] & M \\ [-15]
 &(g_2, m) \ar[r,mapsto,"d_0" description, near start]  &m \\
(g_1, g_2, m) \ar[ur, mapsto, "d_0" description, near start] \ar[r,mapsto, "d_1" description, pos=.4] \ar[dr,mapsto, "d_2" description, near start] & (g_1 g_2, m)  \ar[ur,mapsto,"d_0" description, near start, bend right=15]  &  g_2 m \\
& (g_1, g_2 m) \ar[ur,mapsto,"d_0" description, near start, bend right=10] \ar[r,mapsto,"d_1" description, pos=.4]& g_1 g_2 m
\ar[from=2-2, to=3-3, mapsto, crossing over,"d_1" description, near start, bend left=10] \ar[from=3-2, to=4-3, mapsto, crossing over, "d_1" description, near start, bend left=15]
\end{tikzcd} \]
%\[ \includegraphics{a5.pdf}\]

\begin{prop}\label{Prop:EquivStruct}For $M \in G\GMan$, an equivariant bundle $Q \in G\GBun[K](M)$ is equivalent to: $Q \in \Bun[K](M)$, together with a bundle morphism ${\Psi\colon d_0^*Q \xrightarrow{\iso} d_1^*Q}$ satisfying the cocycle identity \eqref{Eq:CocycleIdentity}.
\end{prop}
\begin{proof}

For $Q \in \Bun[K](M)$, a map of $K$-bundles $G \times Q \xrightarrow{\ell} Q$ is equivalent to a bundle isomorphism $d_0^* Q \xrightarrow{\Psi} d_1^* Q$, as indicated in the following diagram.
\begin{equation}\label{eq:PsiDefn} \begin{tikzcd}[column sep=tiny]
G \times Q  \ar[rrr, bend left=25, "\ell", dashed,swap] \ar[dr] \ar[r, phantom, "\iso"] & d_0^*Q \ar[r, dashed, "\Psi"] \ar[d] & d_1^* Q \ar[dl] \ar[r, "\wt{d}_1"]& [50] Q \ar[d] \\
&G\times M \ar[rr, "d_1"] && M 
\end{tikzcd} \end{equation}
%\begin{equation}\label{eq:PsiDefn} \raisebox{-0.42\height}{\includegraphics{a6.pdf}} \end{equation}
By regarding $d_i^*Q \subset G \times M \times Q$ we have $\Psi(g, m, q) = (g, m, \ell(g,q))$;  under the natural identification  $d_j^* d_i^* Q \iso (d_i \o d_j)^*Q \subset G \times G \times M \times Q$ one sees that 
\[ \left(d_i^*\Psi\right) (g_1, g_2, m, q) = \begin{cases} (g_1, g_2, m, \ell(g_2, q)) & i = 0, \\ (g_1, g_2, m, \ell(g_1 g_2, q)) & i = 1, \\ (g_1, g_2, m, \ell(g_1, q)) &i=2. 
\end{cases} \]
Hence, $\ell$ satisfies the associativity condition $\ell(g_1g_2, q) = \ell(g_1, \ell(g_2,q))$ if and only if the diagram 
\begin{equation}\label{Eq:CocycleIdentity} \begin{tikzcd}[row sep=.8]
 &[-20] d_1^* d_0^* Q \ar[rr, "d_1^*\Psi"] &&d_1^* d_1^* Q \ar[dr,phantom,"\iso" rotate=-45] \\
 (d_0 \o d_0)^*Q \ar[ur,phantom,"\iso" rotate=45] \ar[dr,phantom,"\iso" rotate=-45] &[-35]&&& [-20] (d_1 \o d_1)^* Q  \\
& d_0^* d_0^* Q \ar[r, "d_0^*\Psi"] & d_0^* d_1^* Q \iso d_2^* d_0^* Q \ar[r, "d_2^* \Psi"] &d_2^* d_1^* Q  \ar[ur,phantom,"\iso" rotate=45]
\end{tikzcd} \end{equation}
%\begin{equation}\label{Eq:CocycleIdentity} \raisebox{-0.42\height}{\includegraphics{a7.pdf}} \end{equation}
commutes.  We may abbreviate this cocycle condition as $d_1^*\Psi \iso d_2^* \Psi \o d_0^*\Psi$.
\end{proof}

An important subtlety arises when introducing connections, as the isomorphism $\Psi$ does not usually preserve the induced connections; one must instead subtract a counter-term.  This term was explicitly used in \cite{Redden16a}, but it appeared implicitly in \cite{MR1850463} and is closely related to the equivariant curvature form \cite{MR705039}.  

Remember that for $Q \in \Bun[K](M)$, the space of connections is affine over $\O^1(M;\fk_{Q})$, where $\fk_Q = Q \times_{\Ad} \fk$ is the associated adjoint bundle.  Suppose that $\rho\colon \fg \to \Gamma(TQ)^K$ is a linear map (e.g. $\rho$ is induced by a $G$-action as defined in \eqref{Eq:VF}).  Explicitly, the requirement that $\rho$ maps to $K$-invariant vector fields  can be written
\[ \rho_{qk}(A) = \rho_q(A) \mycdot k  \in TQ_{qk} ,\]
where $v\mycdot k$ denotes the image of a tangent vector $v$ under the map between tangent spaces induced by right mutiplication of $k \in K$.  Given $\alpha \in \O^1(X;\fg)$, $f\colon X\to M$, and a connection $\Theta$ on $Q$, we define the form 
\[ \left(  \iota_{\rho(\alpha)} f^*\Theta \right) \, \in \, \O^1(X; \fk_{f^*Q}) = C^\infty\big(X, TX^* \otimes (f^*Q \times_{\Ad} \fk) \big) \] 
by the following equation.  Given $x \in X$, $v \in TX_x$, and $q \in (f^*Q)_x$ a point in the fiber of $x$, then 
\begin{equation}\label{Eq:Contract}
\big(\iota_{\rho(\alpha)} f^*\Theta\big)_{x}(v) := \Big[ q, \, \Theta\big( \rho_{_{f(q)}} (\alpha(v)) \big)  \Big] \in f^*Q \times_{\Ad} \fk.
\end{equation}
(Though it is a minor abuse of notation, we use the same symbol here for the base maps and pullback maps to minimize notational confusion.)  Subtracting this term from the pullback connection $f^*\Theta$ gives the modified connection $(1-\iota_{\rho(\alpha)})f^*\Theta$ on $f^*Q$.  The functoriality of this construction can be expressed as the equation 
\begin{equation}\label{Eq:ContractionPullback} \varphi^*(1-\iota_{\rho(\alpha)}) = (1-\iota_{\rho(\varphi^*\alpha)}) \varphi^*.
\end{equation}
More precisely, given $\alpha\in \O^1(X;\fg)$ and maps
\[ \begin{tikzcd}
\varphi^* f^*Q \ar[d] \ar[r, "\varphi"]& f^*Q \ar[d] \ar[r, "f"]& Q \ar[d] \\
Y \ar[r,"\varphi"]& X \ar[r, "f"]& M,
\end{tikzcd} \]
%\[\includegraphics{a8.pdf} \]
then $\varphi^*\left((1-\iota_{\rho(\alpha)})f^*\Theta \right) = (1-\iota_{\rho(\varphi^*\alpha)}) (f \o \varphi)^*\Theta \in \O^1(Y; \fk_{\varphi^*f^*Q})$.  

\begin{lemma}\label{Lem:Contract}The term $\iota_{\rho(\alpha)}$ is well-defined and satisfies \eqref{Eq:ContractionPullback}.
\end{lemma}
\begin{proof}
Both statements follow directly from the definitions.  We first check that $\iota_{\rho(\alpha)}$, as given  in \eqref{Eq:Contract}, is well-defined.  For $x \in X$, let $q$ and $qk$ be two points in the fiber $(f^*Q)_x$.  Then, for $v\in TX_x$, 
\[
 \Theta\Big( \rho_{_{f(qk)}} \big(\alpha(v) \big) \Big)  =   \Theta\Big( \rho_{_{f(q)k}} \big(\alpha(v) \big) \Big)  
 =  \Theta\Big( \rho_{_{f(q)}} \big(\alpha(v)  \big) \mycdot k \Big)  
 = \Ad_{k^{-1}}  \Theta\Big( \rho_{_{f(q)}} \big(\alpha(v) \big) \Big) .
 \]
Hence,
 \[ \Big[qk, \,  \Theta\big( \rho_{_{f(qk)}} (\alpha(v) ) \big) \Big] = \Big[qk, \,  \Ad_{k^{-1}}  \Theta\big( \rho_{_{f(q)}} (\alpha(v) ) \big)  \Big] = \Big[ q, \, \Theta\big( \rho_{_{f(q)}} (\alpha(v)) \big)  \Big]\]
 is a well-defined element of  $f^*Q \times_{\Ad}\fk$.

To check \eqref{Eq:ContractionPullback}, let $v \in TY_{y}$ and $q \in (\varphi^* f^*Q)_y$.  The form $\varphi^*(\iota_{\rho(\alpha)}f^*\Theta)$, when evaluated on $v$, is defined as the unique element in $(\varphi^* f^*Q \times_{\Ad} \fk)_{y}$ that maps to $(\iota_{\rho(\alpha)}f^*\Theta)(\varphi_*v) \in (f^*Q \times_{\Ad} \fk)_{\varphi(y)}$.  We have
\[ \left( \iota_{\rho(\alpha)} f^*\Theta \right)_{\varphi(y)}(\varphi_* v)  =  \left[\varphi(q), \Theta\left(  \rho_{_{f(\varphi(q))}} (\alpha(\varphi_*(v)) \right) \right] = [\varphi(q), \Theta \left( \rho_{_{f(\varphi(q))}} (\varphi^*\alpha(v))   \right)].  \]
Hence,
\[
\left( \varphi^*(\iota_{\rho(\alpha)} f^*\Theta) \right)_y(v)
 = [q,\Theta \left( \rho_{_{f(\varphi(q))}} (\varphi^*\alpha(v))   \right) ] 
=  \left( \iota_{\rho(\varphi^* \alpha)} ( \varphi^* f^* \Theta ) \right)_y(v).
\]
\end{proof}

Of particular importance is when $\alpha$ is the Maurer--Cartan 1-form $\theta_\fg \in \O^1(G;\fg)$.  We also use the same symbol $\theta_\fg \in \O^1(G \times M;\fg)$ to denote its pullback via the projection map $G\times M \to G$.  Assume now that $Q \in G\GBun[K](M)$, and that $\Psi\colon d_0^*Q \to d_1^*Q$ and  $\rho \colon \fg \to \Gamma(TQ)^K$ are defined by the $G$-action.  Then, we have the following characterization of $G$-invariant and $G$-basic connections on $Q$.

\begin{prop}\label{Prop:EquivConns}For $Q \in G\GBun[K](M)$, a connection $\Theta$ is $G$-invariant if and only if ${\Psi^*d_1^*\Theta = (1-\iota_{\rho(\theta_\fg)})d_0^*\Theta}$.  A connection is $G$-basic if and only if it is $G$-invariant and $\Psi^*d_1^*\Theta = d_0^*\Theta$.
\end{prop}
\begin{proof}Given a connection $\Theta \in \O^1(Q;\fk)$, we use the natural isomorphisms in \eqref{eq:PsiDefn} to calculate $\Psi^*d_1^*\Theta$ as $\ell^*\Theta \in \O^1(G\times Q;\fk)$.  Let $A \in TG_e$ generate the vector field $\xi \in \fg$, with $\xi_g = g \mycdot A \in TG_g$, and let $v \in TQ_q$.  Then,
\[ \begin{tikzcd}[column sep=-7, row sep=-1]
TG_g &\oplus &TQ_q \ar[r, "\ell_*"]&[25]TQ_{gq} \\
&&v \ar[r, mapsto] &g\mycdot v \\
g\mycdot A \ar[rrr,mapsto] &&& - g \mycdot \rho(A) ,
\end{tikzcd}\]
%\[ \includegraphics{a9.pdf}\]
with the negative sign due to the convention explained in \eqref{Eq:VF}.  Hence, 
\begin{align*}
(\ell^* \Theta)_{(g, q)}  (g\mycdot A + v) &= \Theta_{g q}  \left( g\mycdot (v-\rho(A)) \right) 
	= (g^*\Theta)_{q} \left( v -\rho(A) \right). 
\end{align*}
On the other hand, $d_0\colon G\times M \to M$ is the projection map with $d_{0*}(g\mycdot A + v) = v$, and $\theta_\fg(g\mycdot A + v) = A$.  Hence, $(1-\iota_{\rho(\theta_\fg)})d_0^*\Theta \in \O^1(G\times Q,\fk)$ is given by 
\begin{align*}
\big((1-\iota_{\rho(\theta_\fg)})d_0^*\Theta \big)_{(g,q)} (g \mycdot A + v) &=  \Theta_{q}\Big( d_{0*} (g\mycdot A + v) - \rho_q(\theta_\fg(g \mycdot A+v)) \Big) \\
&= \Theta_{q}\left(v-\rho(A)\right).
\end{align*}
Hence, $\ell^*\Theta = (1-\iota_{\rho(\theta_{\fg})} ) d_0^* \Theta$ if and only if $g^*\Theta = \Theta$ for all $g\in G$.  Furthermore, $\ell^*\Theta = d_0^*\Theta$ if and only if $\Theta$ is $G$-invariant and $\Theta(\rho(A)) = 0$ for all $A\in TG_e$.
\end{proof}

%%%%%%%%%%%%%%%%%%%%%%%%%%%%%%%%%%%%%%%%%%%%%%%%%%%
%%%%%%%%%%%%%%%%%%%%%%%%%%%%%%%%%%%%%%%%%%%%%%%%%%%
\section{Connections on the differential Borel quotient stack}\label{Sec:Equivalences}

We proceed to outline our main result, Theorem \ref{Thm:EquivBun}.  The actual proof will follow from a series of smaller lemmas.  Consider the following diagram of functors.
\begin{equation}\label{Eq:ProofDiagram} \begin{tikzcd}[column sep=tiny]
&G\GBunc[K](M) \ar[dr, "\Tn_1"] \\
\Bunc[K](M \nslash G) \ar[ur, "\Tn_3"]&&\Bunc[K](\En G \times_G M) \ar[ll, "\Tn_2"]
\end{tikzcd} \end{equation}
%\begin{equation}\label{Eq:ProofDiagram} \raisebox{-0.5\height}{\includegraphics{a10.pdf}} \end{equation}
There is a similar diagram, with all $\nabla$'s removed, given by omitting connections.

The functor 
\[ G\GBunc[K](M) \xrightarrow{\Tn_1} \Bunc[K](\En G \times_G M)\]
was given in \cite[\textsection 5]{Redden16a} by the following construction.  For $(Q,\Theta_Q) \in G\GBunc[K](M)$, let $\Tn_1(Q,\Theta_Q) = \cQn$ denote the induced map $\En G \times_G M \to \Bn K$.  For $X \xrightarrow{(P,\Theta_P,F)} \En G \times_G M$, which is a $G$-bundle with connection $(P,\Theta) \to X$ and $G$-equivariant map $F\colon P\to M$, it is defined
\begin{equation}  \cQn(P,\Theta_P, F) \= \Big((F^*Q)/G, \ (1-\iota_{\rho(\Theta_P)})F^*\Theta_Q \Big) \in \Bunc[K](X). \end{equation}
Here, $F^*Q$ is a $(G\times K)$-bundle over $X$, with the $G$-structure induced by pulling back along $G$-equivariant maps.  The form $(1-\iota_{\rho(\Theta_P)})F^*\Theta_Q \in \O^1(F^*Q;\fk)$ is $G$-basic, and hence it naturally descends to $\O^1 ((F^*Q)/G;\fk )$.  The map $\Tn_2$ is induced by the natural map $M \nslash G \to \En G \times_G M$, as described in Section \ref{Sec:Review}.  

The main work is to construct the functor $\Tn_3$.  In particular, we must show that the construction outlined  in Section \ref{Sec:Introduction} actually gives rise to a well-defined functor with the desired properties. 

\begin{thm}\label{Thm:EquivBun}
Each of the functors $\Tn_i$ in \eqref{Eq:ProofDiagram} is an equivalence of categories.  Also, the same result holds when every symbol $\nabla$ is removed from the diagram, giving equivalences
\[ G\GBun[K](M) \iso \Bun[K](\E G \times_G M) \iso \Bun[K](M\sslash G). \]
Furthermore, $\cQn \in \Bunc[K](M\nslash G)$ descends to $\Bunc[K](M\sslash G)$ if and only if the corresponding object $(Q,\Theta) \in G\GBunc[K](M)$ is $G$-basic, giving equivalences 
\[ G_{\bas} \GBunc[K](M) \iso \Bunc[K](\E G \times_G M) \iso \Bunc[K](M\sslash G).\]
\end{thm}

We now proceed to develop the necessary lemmas used in proving Theorem \ref{Thm:EquivBun}.  As noted in Remark \ref{Rem:Stacks}, a stronger version of Lemma \ref{Lem:T2Faithful} also follows easily from standard properties of stacks.\\

\noindent
\begin{minipage}{.7\textwidth}
\begin{lemma}\label{Lem:Cat}Suppose we have functors $T_i$ between categories $\cC_i$ 
with $T_3 T_2 T_1 \iso \Id_{\cC_1}$, and $T_2, T_3$ faithful.  Then each $T_i$ is an equivalence of categories.
\end{lemma}
\end{minipage}
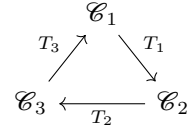
\begin{minipage}{.3\textwidth}
\[ \begin{tikzcd}[column sep=tiny]
&\cC_1 \ar[dr, "T_1"] \\
\cC_3 \ar[ur, "T_3"]&&\cC_2 \ar[ll, "T_2"]
\end{tikzcd} \]
%\[ \includegraphics{a11.pdf} \]
\end{minipage}
\begin{proof}Since $T_2$ and $T_3$ are faithful, then $(T_3 T_2)$ is faithful.  Since $(T_3 T_2)(T_1) \iso \Id_{\cC_1}$, then $(T_3 T_2)$ is full and essentially surjective.  Hence, $(T_3 T_2)$ and $T_1$ are equivalences.  By rearranging parentheses, the same argument shows that $T_3$ and $(T_2 T_1)$ are equivalences.  Finally, that $T_3$ and $T_1$ are equivalences implies that $T_2$ is an equivalence.
\end{proof}

\begin{lemma}\label{Lem:T2Faithful}The functors $T_2$ and $\Tn_2$ are faithful.
\end{lemma}
\begin{proof}First note that the pullback functor $\Bunc[K](X) \xrightarrow{f^*}\Bunc[K](Y)$, induced by $f\colon Y\to X$, is faithful when $f$ is surjective.  Now let $\cQn, \cQn' \in \Bunc[K](\En G \times_M G)$, with $\cQn \xrightarrow{\psi_i} \cQn'$ such that $\Tn_2 \psi_1 =\Tn_2 \psi_2$.   We proceed to prove that $\Tn_2$ is faithful by directly showing that $\psi_1=\psi_2$.

Consider an arbitrary element $(P,\Theta,F) \in  (\En G \times_G M)(X)$.  For any bundle $P\xrightarrow{\pi} X$, the pullback $\pi^* P \to P$ has a canonical trivialization $\pi^*P \iso P \times G$ induced by the identity map $P\to P$.  Therefore, $\pi^*(P,\Theta,F)$ is canonically isomorphic to an object in the image of $(M\nslash G)(P)$,
\[ \begin{tikzcd}[row sep=-1] 
(\En G \times_G M) (X) \ar[r, "\pi^*"] & (\En G \times_G M) (P) & \ar[l] (M\nslash G) (P) \ar[l] \\
(P,\Theta, F) \ar[mapsto, r] & \pi^*(P,\Theta, F) \iso  & (\Theta,F). \ar[l, mapsto]
\end{tikzcd} \]
%\[ \includegraphics{a12.pdf} \]
The induced isomorphisms $\pi^*( \cQn^{(\prime)} (P,\Theta, F))  \iso (\Tn_2 \cQn^{(\prime)} )(\Theta,F) \in \Bunc[K](P)$ naturally identify the morphisms $\pi^*(\psi_i(P,\Theta,F))$ with $(\Tn_2 \psi_i)(\Theta,F)$.  Hence, \[ \pi^*(\psi_1(P,\Theta,F)) = \pi^*(\psi_2(P,\Theta,F)).\]  Since $\pi$ is surjective, then $\pi^*$ is faithful and $\psi_1(P,\Theta,F) = \psi_2(P,\Theta,F)$ for every $X \xrightarrow{(P,\Theta,F)} \En G \times_G M$, and hence $\psi_1=\psi_2$.  The same argument, with the connections omitted, shows that $T_2$ is faithful.  
\end{proof}

We first ignore connections and construct the functor $T_3$, which is given by evaluating $\cQ \in \Pre_\Gpd(M\sslash G, \B K)$ on the Lie groupoid $G\times M \rightrightarrows M$.  We will sketch the construction using diagrams before providing explicit details within the proof of Proposition \ref{Prop:T3}.   Under the natural Yoneda embedding, which identifies an object in $\cM(X)$ with a map $X \to \cM$, the bundle $Q \in \Bun[K](M)$ is defined by the composition
\[ \begin{tikzcd} M \ar[r, "\Id_M"] & M \sslash G \ar[r, "\cQ"] & \B K, \end{tikzcd} \]
and the $G$-structure on $Q$ is given by the isomorphism $\cQ(\Id_{G\times M})$, as shown in
\[  \begin{tikzcd}
G\times M  & [30] M \sslash G \ar[r, "\cQ"] & \B K.
 \ar[from=1-1, to=1-2, bend left=30, "" '{name=U}, "{d_0}"]   \ar[from=1-1, to=1-2, bend right=30,  "" {name=D}, "{d_1}" '] 
\ar[Rightarrow, from=U, to=D, "\Id_{G\times M}" ] 
\end{tikzcd} \]
%\[  \includegraphics{a13.pdf} \]

%\[ \begin{tikzcd} G\times G \times M \ar[r,"d_1" description] \ar[r, shift left=2, "d_0"] \ar[r, shift right=2, "d_2" '] & G \times M \ar[r, shift left, "d_0"] \ar[r, shift right, "d_1" '] & M \ar[r, "\Id_M"] & M \sslash G \ar[r, "\cQ"] & \B K  \end{tikzcd} \]
% and the cocycle condition \eqref{Eq:CocycleIdentity} is given by the natural associative identity
%\[  \begin{tikzcd}
%G\times G\times M  & [40] M \sslash G \ar[r,"\cQ"] &\B K.
% \ar[from=1-1, to=1-2, bend left=50, "" '{name=U}, start anchor = north east, end anchor = north west]  
% \ar[from=1-1, to=1-2,  "" {name=MU}, "" ' {name=MD},  start anchor = east, end anchor =  west] 
%   \ar[from=1-1, to=1-2, bend right=80,  "" {name=D},  start anchor = south east, end anchor = south west] 
%\ar[Rightarrow, from=U, to=MU, "d_0" ' ]   \ar[Rightarrow, from=MD, to=D, "d_2" ' ]  \ar[Rightarrow, from=U, to=D, "d_1", bend left=70, crossing over, pos=.55]
%\end{tikzcd} \]

\begin{prop}\label{Prop:T3}The functor $T_3\colon \Bun[K](M\sslash G) \to G\GBun[K](M)$ is well-defined.
\end{prop}
\begin{proof}We first explicitly define $T_3$ at the level of objects.  Let $\cQ \in \Bun[K] (M\sslash G) = \Pre_\Gpd (M\sslash G, \B K)$.  Using $\Id_M \in (M\sslash G)(M)$, we obtain the $K$-bundle
\[ Q = T_3(\cQ)  \=  \cQ(\Id_M) \in \B K(M) = \Bun[K](M).\]
As explained in Proposition \ref{Prop:EquivStruct}, lifting $Q$ to $G\GBun[K](M)$ is equivalent to defining an isomorphism $\Psi\colon d_0^*Q \to d_1^*Q$ satisfying the cocycle identity \eqref{Eq:CocycleIdentity}.    The identity map $\Id_{G\times M}$ is a morphism from $d_0$ to $d_1$ in $(M\sslash G)(G\times M)$; applying $\cQ$ and using appropriate natural isomorphisms, we define $\Psi$ in $\Bun[K](G\times M)$ via
\[ \begin{tikzcd}[column sep=tiny]
d_0^*Q \ar[r, phantom, "="] \ar[d, "\Psi" '] & d_0^* \cQ(\Id_M) \ar[r, phantom, "\iso"] & \cQ(\Id_M \o \, d_0) \ar[r, phantom, "="] & \cQ(d_0)  \ar[d, "{\cQ(\Id_{G\times M})}"] \\
d_1^*Q \ar[r, phantom, "="] & d_1^* \cQ(\Id_M) \ar[r, phantom, "\iso"] & \cQ(\Id_M \o \, d_1) \ar[r, phantom, "="] & \cQ(d_1).
\end{tikzcd}\]

We now verify the cocycle identity  $d_1^* \Psi \iso d_2^*\Psi \o d_0^* \Psi$.  Consider the morphisms  $d_0 \o d_i \xrightarrow{d_i} d_1 \o d_i$ in $(M\sslash G)(G\times G \times M)$, and let $\bullet$ denote the composition of such morphisms.  Then the diagram 
\[ \begin{tikzcd}[row sep=small] d_0 \o d_0 = d_0 \o d_1 \ar[dr, start anchor={[xshift=-50]}, end anchor = north west, "d_0" '] \ar[rr, "d_1"] && d_1 \o d_1 = d_1 \o d_2 \\
& d_1 \o d_0 = d_0 \o d_2 \ar[ur, start anchor=north east, end anchor = {[xshift=50]}, "d_2" ']
\end{tikzcd} \]
commutes, as evidenced by the calculation
\[ (d_2\bullet d_0)(g_1, g_2, m) = (g_1, g_2m) \bullet (g_2, m) = (g_1 g_2, m) = d_1(g_1, g_2, m),\]
giving $d_2 \bullet d_0 = d_1$.  Under the natural isomorphisms
\[ \begin{tikzcd}[column sep=small]
d_i^* d_0^* Q \ar[r, phantom, "\iso"] \ar[d, "d_i^*\Psi" '] & d_i^* \cQ(d_0) \ar[r, phantom, "\iso"] \ar[d, "d_i^* {\cQ(\Id_{G\times M})}" ] & \cQ(d_0 \o d_i) \ar[d, "\cQ(d_i)" ] \\
d_i^* d_1^* Q \ar[r, phantom, "\iso"] & d_i^* \cQ(d_1) \ar[r, phantom, "\iso"] & \cQ(d_1 \o d_i), 
\end{tikzcd} \]
we see that $d_i^*\Psi \iso \cQ(d_i)$.  The desired cocycle identity then follows from
\[ d_1^*\Psi \iso \cQ(d_1) = \cQ(d_2 \bullet d_0) = \cQ(d_2) \o \cQ(d_0) \iso d_2^*\Psi \o d_0^* \Psi.\]
Hence, $Q$ is a well-defined object in $G\GBun[K](M)$.

Because our construction of $T_3(\cQ)$ involved no choices, the extension of $T_3$ to morphisms follows immediately.  Explicitly, suppose that $\varphi \colon \cQ \to \cQ'$ is a morphism in $\Pre_\Gpd( M\sslash G, \B K)$.  When applied to the test manifold $M$, this natural transformation gives an isomorphism $\phi = T_3(\varphi) \= \varphi(\Id_M)$ in $\Bun[K](M)$ by
\[ Q = T_3(\cQ) = \cQ(\Id_M) \xrightarrow{\varphi(\Id_M)} \cQ'(\Id_M) = T_3(\cQ) = Q'. \]
Applying the natural transformation $\varphi$ on $G\times M$ gives the commutative diagram
\[ \begin{tikzcd}[column sep=tiny]
d_0^* Q \ar[r, phantom, "\iso"] \ar[d, "d_0^* \phi" '] & \cQ(d_0) \ar[d, "\varphi(d_0)" '] \ar[r, "{\cQ(\Id_{G\times M})}"] & [40] \cQ(d_1)  \ar[d, "\varphi(d_1)"] \ar[r, phantom, "\iso"] & d_1^* Q \ar[d, "d_1^*\phi" ]\\
d_0^* Q' \ar[r, phantom, "\iso"] & \cQ'(d_0) \ar[r, "{\cQ'(\Id_{G\times M})}" '] & \cQ'(d_1) \ar[r, phantom, "\iso"] & d_1^* Q' , 
\end{tikzcd} \]
which implies that $\phi$ is equivariant with respect to the $G$-structures on $Q, Q'$ induced by $\cQ, \cQ'$ respectively.
%\[ \begin{tikzcd} (M\sslash G)(M)  &  \B K (M),
% \ar[from=1-1, to=1-2, bend left=30, "" '{name=U}, "{\cQ}"]   \ar[from=1-1, to=1-2, bend right=30,  "" {name=D}, "{\cQ'}" '] 
%\ar[Rightarrow, from=U, to=D, "\varphi" ] 
%\end{tikzcd} \]
\end{proof}

Suppose now that $\cQn \in \Bunc[K](M\nslash G)$, and define the bundle with connection
\[(Q,\Theta) \= \cQn(0, \Id_M) \in \Bunc[K](M).\]
Let $\cQ \in \Bun[K](M\sslash G)$ denote the composition
\[  M \nslash G \overset{\cQn}\longrightarrow \Bn K \longrightarrow \B K,\]
given by forgetting the connection.  In particular $Q = \cQ(0,\Id_M)$, in agreement with our earlier definition of $Q$.  We now must determine how the $\O^1(\fg)$ part of $M\nslash G$ affects the resulting bundles $\cQn$.  

\begin{lemma}\label{Lem:Affine}Let $\cQn \in \Bunc[K](M\nslash G)$.  For $X \xrightarrow{(\alpha, f)} \O^1(\fg) \times M$, there are canonical isomorphisms 
\begin{align*} \cQn(0, f) \iso f^*\cQn(0,\Id_M) = f^*(Q,\Theta) &\in \Bunc[K](X), \\ 
\cQ(\alpha, f) \iso \cQ(0,f) \iso f^*Q &\in \Bun[K](X).  \end{align*}
\end{lemma}
\begin{proof}
For the first isomorphism, note that any map $X \xrightarrow{(0,f)} M\nslash G$ factors through $M \xrightarrow{(0,\Id_M)} M\nslash G$.  Hence, $\cQn\big( (0,f) \big) = \cQn\big( (0,\Id_M) \o f \big) \iso f^*\cQn(0,\Id_M)$.  

The second property follows from the space of connections on a given bundle being affine.  Since $\O^1(\fg)$ is a sheaf of vector spaces, the desired isomorphism follows by connecting the forms $0, \alpha \in \O^1(X;\fg)$ along a path as follows.
\[ \begin{tikzcd}
X \ar[d, hook, "\Id_X \! \times \{0\}" ']  \ar[dr, "{(0,f)}"]\\
X \times [0,1]  \ar[r, "{(t\alpha, \, f \o \pi)}"] &[15]\O^1(\fg)\times M \ar[r, "\cQn"] &\Bn K\\
X \ar[u, hook', "\Id_X \! \times \{1\}"]  \ar[ur, "{(\alpha,f)}" ']
\end{tikzcd} \]
%\[\includegraphics{a16.pdf} \]
This induces a bundle with connection $\cQn(t\alpha, f \o \pi) \to X\times [0,1]$.  Parallel transport along $[0,1]$ defines the canonical isomorphism $\cQ(0,f) \iso \cQ(\alpha, f)$.
\end{proof}

\begin{lemma}\label{Lem:ForgetConns}The functor $\Bunc[K](M\nslash G) \to \Bun[K](M\nslash G)$, given by $\cQn \mapsto \cQ$, descends to a functor $\Bunc[K](M\nslash G) \to \Bun[K](M\sslash G)$.
\end{lemma}
\begin{proof}Let $M\nslash G \xrightarrow{\cQn} \Bn K$.  At the level of objects and morphisms, respectively, we use the maps
\[ M \xrightarrow{(0,\, \Id_M)} \O^1(\fg) \times M, \qquad G \times M \xrightarrow{(\Id_G,\, \theta_\fg, \, \Id_M)} G \times \O^1(\fg) \times M.\]
For $X\xrightarrow{f} M$, define $\cQ(f) = \cQ(0,f)$, and for $X \xrightarrow{(g,f)} G\times M$, define $\cQ(g,f)$ by $\cQn(g, g^*\theta_\fg, f)$, which gives an isomorphism $\cQn(g^*\theta_\fg,f) \to \cQn(0, gf).$  Combined with the natural isomorphism from Lemma \ref{Lem:Affine}, this gives the desired morphism
\[ \cQ(f) = \cQ(0, f) \iso \cQ(g^*\theta_\fg, f) \xrightarrow{\cQn(g,\, g^*\theta_g, \,f)} \cQ(0, gf) = \cQ(gf).\]
\end{proof}

Proposition \ref{Prop:T3} and Lemma \ref{Lem:ForgetConns} combine to define a functor \[ \Bunc[K](M\nslash G) \longrightarrow G\GBun[K](M).\]
Therefore, for $\cQn \in \Bunc[K](M\nslash G)$, we know that $Q$ has a $G$-equivariant structure, but we must show that $\Theta$ is $G$-invariant.

\begin{lemma}\label{Lem:GinvtConn}For $\cQn \in \Bunc[K](M \nslash G)$, the bundle $(Q,\Theta)\in G\GBunc[K](M)$.
\end{lemma}
\begin{proof}Fix $h\in G$, regarded also as a map $h\colon M\to G$.  Since $h$ is a constant map, $0\mycdot h = 0 \in \O^1(M;\fg)$.  Therefore, we have the morphism
\begin{equation}\label{Eq:Tn3Gaction} \begin{tikzcd}
M  & [60] M \nslash G,
 \ar[from=1-1, to=1-2, bend left=30, "" '{name=U}, "{(0,\,\Id_M)}" pos=.51]   \ar[from=1-1, to=1-2, bend right=30,  "" {name=D}, "{(0,\, h\cdot \Id_M )}" ' pos=.51] 
\ar[Rightarrow, from=U, to=D, "{(h, 0, \Id_M)}" ] 
\end{tikzcd} \end{equation}
%\begin{equation}\label{Eq:Tn3Gaction}\raisebox{-0.5\height}{ \includegraphics{a17.pdf}} \end{equation}
with the map $\ell_h\colon Q \to Q$ induced by $\cQ(h, 0, \Id_M)$.  This combines with the natural isomorphisms from Lemma \ref{Lem:Affine} to give
\[  (Q,\Theta) \iso \cQn(0, \Id_M) \xrightarrow{\cQn(h, 0, \Id_M)} \cQn(0, h \cdot \Id_M) \iso h^*(Q,\Theta), \]
which is a morphism in $\Bn K(M)$.  Hence, $\ell_h^*\Theta=\Theta$ and $\Theta$ is $G$-invariant.
\end{proof}

\begin{lemma}\label{Lem:UniversalFormula}There is a natural isomorphism $\cQn(\alpha, f) \iso (1-\iota_{\rho(\alpha)})f^*(Q,\Theta)$.
\end{lemma}
\begin{proof}
First consider 
\begin{equation*} \begin{tikzcd}
G\times M  & [60] M \nslash G,
 \ar[from=1-1, to=1-2, bend left=20, "" '{name=U}, "{(\theta_\fg,\, d_0)}"]   \ar[from=1-1, to=1-2, bend right=20,  "" {name=D}, "{(0,\, d_1)}" '] 
\ar[Rightarrow, from=U, to=D, "{(g, \theta_\fg, d_0)}" ] 
\end{tikzcd} \end{equation*}
%\begin{equation*} \includegraphics{a18.pdf} \end{equation*}
where $G\times M\xrightarrow{g} G$ is the projection map, and $0 \mycdot g = \theta_\fg$.  This induces the following morphism of $K$-bundles with connection on $G\times M$, which defines $\Psi \colon d_0^*Q \to d_1^*Q$ after forgetting the connections.
\[ \begin{tikzcd}
\cQn(\theta_\fg, d_0) \ar[dr] \ar[rr, "{\cQn(g,\theta_\fg,d_0) \iso \Psi}"] && \cQn(0,d_1) \iso d_1^*(Q,\Theta) \ar[dl] \\ 
&G\times M 
\end{tikzcd}\]
%\[   \includegraphics{a19.pdf} \]
Lemma \ref{Lem:GinvtConn} states that $\Theta$ is $G$-invariant, and Proposition \ref{Prop:EquivConns} states that $\Psi^*d_1^*\Theta = (1-\iota_{\rho(\theta_\fg)})d_0^*\Theta$ for $G$-invariant connections.  Hence, 
\begin{equation}\label{Eq:UnivFormula} \cQn(\theta_\fg, d_0) \iso (1-\iota_{\rho(\theta_\fg)}) d_0^*(Q,\Theta).\end{equation}

When $\alpha = g^*\theta_\fg \in \O^1(X;\fg)$ for some $g\colon X \to G$, then $(\alpha, f)$ factors as
\[ \begin{tikzcd}
&G \times M \ar[dr, "{(\theta_\fg, d_0)}"] \\
X \ar[ur, "{(g,f)}"] \ar[rr, "{(\alpha, f)}"] && \O^1(\fg) \times M.
\end{tikzcd} \]
%\[ \includegraphics{a20.pdf} \]
Using formulas \eqref{Eq:ContractionPullback} and \eqref{Eq:UnivFormula},
\[ \cQn(\alpha, f) \iso (g,f)^* (1-\iota_{\rho(\theta_\fg)}) d_0^*(Q,\Theta) \iso (1-\iota_{\rho(g^*\theta_\fg)}) f^*(Q,\Theta).\]

In general, $\alpha \in \O^1(X;\fg)$ will not equal $g^*\theta_\fg$ for some map $g\colon X \to G$.  But, any bundle with connection on $X$ is completely determined by its restriction to all paths in $X$.  And, on any path $\gamma \colon I \to X$ it is true that $\gamma^*\alpha = g^*\theta_\fg$ for some $g\colon I \to G$.  Therefore,  
\[ \gamma^*\cQn(\alpha, f) \iso \cQn( g^*\theta_\fg, f\o \gamma) \iso (1- \iota_{\rho(g^*\theta_\fg)})(f \o \gamma)^*(Q,\Theta) = \gamma^*(1-\iota_{\rho(\alpha)}) f^*(Q,\Theta)\]
for all paths $\gamma$.  Hence, $\cQn(\alpha, f) \iso (1-\iota_{\rho(\alpha)}) f^*(Q,\Theta)$.
\end{proof}

\begin{prop}\label{Prop:Tn3Faithful}The construction $\cQn \mapsto \cQn(0,\Id_M)$ defines a faithful functor $\Tn_3$ such that $\Tn_3 \Tn_2 \Tn_1 \iso \Id_{G\GBunc[K](M)}$.  Also, the previously defined $T_3$ is faithful and satisfies $T_3 T_2 T_1 \iso \Id_{G\GBun[K](M)}$.
\end{prop}
\begin{proof}
The functor $\Tn_3$ is well-defined at the level of objects, since Lemma \ref{Lem:GinvtConn} shows that $(Q,\Theta) \= \Tn_3(\cQn) \= \cQn(0,\Id_M)$ is an element of $G\GBunc[K](M)$.   At the level of morphisms, any $\varphi \colon \cQ \to \cQ'$ induces an connection-preserving bundle isomorphism $\phi \colon (Q,\Theta) \to (Q',\Theta')$ in $\Bunc[K](M)$ by $\phi = \Tn_3(\varphi)\= \varphi(0,\Id_M)$.  To see that $\phi$ is $G$-equivariant, apply $\varphi$ to the isomorphism giving multiplication in $Q, Q'$ by fixed $h\in G$, and which was described in Lemma \ref{Lem:GinvtConn} .  We obtain
\[ \begin{tikzcd}[column sep=tiny]
(Q,\Theta) \ar[r, phantom, "="] \ar[d, "\phi" '] & \cQn(0,\Id_M) \ar[d, "{\varphi(0,\Id_M)}" '] \ar[r, "{\cQn(h, 0, \Id_M)}"] & [40] \cQn(0,h\cdot \Id_M)  \ar[d, "{\varphi(0,h\cdot \Id_M)}"] \ar[r, phantom, "\iso"] & h^*(Q,\Theta)\ar[d, "h^*\phi" ]\\
(Q',\Theta') \ar[r, phantom, "="] & \cQn'(0,\Id_M) \ar[r, "{\cQn'(h, 0, \Id_{M})}" '] & \cQn'(0,h\cdot \Id_M) \ar[r, phantom, "\iso"] & h^* (Q',\Theta'),
\end{tikzcd} \]
and hence $\phi$ is $G$-equivariant.  Therefore, $\Tn_3$ is a well-defined functor.

 To see that $\Tn_3$ is faithful, suppose we have two morphisms $\cQn \xrightarrow{\varphi_i} \cQn'$, $i= 1, 2$.  By Lemma  \ref{Lem:UniversalFormula}, for every $(\alpha, f) \in (M\nslash G)(X)$ we have natural isomorphisms
\[ \begin{tikzcd}[row sep=small]
\cQn(\alpha, f) \ar[r, "\varphi_i"] \ar[d, phantom, "\iso" rotate=-90] &[40] \cQn'(\alpha, f)   \ar[d, phantom, "\iso" rotate=-90] \\
(1-\iota_{\rho(\alpha)})f^*(Q,\Theta) \ar[r, "f^* (\Tn_3\varphi_i)"] & (1-\iota_{\rho(\alpha)})f^*(Q',\Theta')
\end{tikzcd} \]
%\[ \includegraphics{a21.pdf} \]
Hence, if $\Tn_3(\varphi_1) = \Tn_3(\varphi_2)$, then $\varphi_1=\varphi_2$.

To check the composition $\Tn_3 \Tn_2 \Tn_1$, we begin with $(Q,\Theta) \in G\GBun[K](M)$.  Then, 
\[ \Tn_3 \big(\Tn_2 \Tn_1(Q,\Theta)\big) = \Big(\Tn_2 \Tn_1 (Q,\Theta)\Big)(0,\Id_M) = \Big(\Tn_1(Q,\Theta)\Big) (M \times G, \theta_\fg, F_{\Id_M}), \]
where $F_{\Id_M}(x,g) = g^{-1} x$.  By the definition of $\Tn_1$ (from \cite[\textsection 5]{Redden16a} and described at the beginning of this section), this is the bundle $\big((F_{\Id_M}^*Q)/G, (1-\iota_{\rho(\theta_\fg)}\big)F_{\Id_M}^*Q) \to M$.  The canonical section $\Id_M\times \, e \colon M \to M\times G$ determines a natural isomorphism 
\begin{align*}
Q &\overset{\iso}\longrightarrow (F_{\Id_M}^*Q)/G \\ 
q &\longmapsto [\pi(q), e, q],\end{align*}
where $\pi \colon Q \to M$ and we regard $F^*Q \subset M\times G \times Q$.  The induced $G$-structure on $\Tn_3 \Tn_2 \Tn_1(Q,\Theta)$, defined by the map $\Psi_h\colon (F_{\Id_M}^*Q)/G \to (F_{h\cdot \Id_M}^*Q)/G$, is induced by the gauge transformation  $(m, g) \mapsto (m, hg)$ on the trivial bundle.  Under our natural isomorphism, the induced map $\ell_h\colon Q \to Q$ is given by the composition
\[ \begin{tikzcd}[row sep=-1, column sep=small]
Q \ar[r, "\iso"]& (F_{\Id_M}^*Q)/G \ar[r, "\Psi_h"] & (F_{h\cdot \Id_M}^*Q)/G \ar[r, "\iso"] & h^*Q \ar[r] & Q \\
q  \ar[r,mapsto] &  {} [\pi(q), e, q]  \ar[r,mapsto] & {}[ \pi(q), h, q] = [\pi(q), 1, h \mycdot  q]  \ar[rr,mapsto]&  & h \mycdot q. 
\end{tikzcd} \]
%\[ \includegraphics{a22.pdf} \]
This is the original $G$-structure on $Q$.  Finally, $(\Id_M \times \, e)^*\theta_\fg = 0$, since the $g$ part of the section is constant.  Therefore,
\[ (\Id_M \times \, e)^*(1-\iota_{\rho(\theta_\fg)})F_{\Id_M}^*\Theta = (1-\iota_{\rho(0)}) (\Id_M \times \, e)^* F_{\Id_M}^* \Theta  \iso \Theta.\]  
Since the isomorphism $(F_{\Id_M}^*Q)/G \iso Q$ is compatible with the $G$-action and connection,  it gives the desired $\Tn_3 \Tn_2 \Tn_1 (Q,\Theta) \iso (Q,\Theta) \in G\GBunc[K](M)$.

The same argument, with connections omitted, shows that $T_3$ is faithful and $T_3 T_2 T_1 \iso \Id_{G\GBun[K](M)}$.
\end{proof}

\begin{proof}[Proof of Theorem \ref{Thm:EquivBun}]From Lemma \ref{Lem:T2Faithful}, the functor $\Tn_2$ is faithful.  From Proposition \ref{Prop:Tn3Faithful}, $\Tn_3$ is faithful, and there is a natural isomorphism $\Tn_3 \Tn_2 \Tn_1 \iso \Id_{G\GBunc[K](M)}$.  Therefore, by Lemma \ref{Lem:Cat}, each $\Tn_i$ is an equivalence of categories.  The same argument also implies that each $T_i$ is an equivalence of categories.

Finally, we know from Lemma \ref{Lem:UniversalFormula} that for $\cQn \in \Bunc[K](M\nslash G)$, 
\[ \cQn(\alpha, f) \iso (1-\iota_{\rho(\alpha)})f^*(Q,\Theta).\]
Hence, $\cQn(\alpha, f) \iso \cQn(\alpha', f)$ for all $(\alpha, f)$ if and only if $\iota_{\rho(A)} \Theta = 0$ for all $A \in \fg$.  Therefore, $\cQn \in \Bunc[K](M\sslash G)$ is well-defined if and only if $\Theta$ is $G$-basic.  The corresponding full subcategory of $\Bunc[K](M\nslash G)$, given by the image of $\Bunc[K](M\sslash G)$, is naturally equivalent to full subcategory of $G$-basic connections $G_{\bas} \GBunc[K](M) \subset G\GBunc[K](M)$.
\end{proof}

%%%%%%%%%%%%%%%%%%%%%%%%%%%%%%%%%%%%%%%%%%%%
%%%%%%%%%%%%%%%%%%%%%%%%%%%%%%%%%%%%%%%%%%%%
\bibliographystyle{alphanum}
\bibliography{MyBibDesk}

\end{document}